\documentclass[12pt]{article}
\usepackage[top=90pt,bottom=90pt,left=90pt,right=90pt]{geometry}

\usepackage{amsmath}
\usepackage{amssymb}
\usepackage{amsthm}
\usepackage{amsfonts}
\usepackage{authblk}
\usepackage{hyperref}
\usepackage{array}

\newtheorem{theorem}{Theorem}[section]
\newtheorem{corollary}[theorem]{Corollary}
\newtheorem{lemma}[theorem]{Lemma}
\newtheorem{definition}[theorem]{Definition}
\newtheorem{example}[theorem]{Example}
\newtheorem{remark}[theorem]{Remark}

\newcommand{\RR}{\mathbb{R}}
\newcommand{\CC}{\mathbb{C}}
\newcommand{\NN}{\mathbb{N}}

\newcommand\keywords[1]{\textbf{Keywords}: #1}

\DeclareMathOperator{\spn}{span}
\DeclareMathOperator{\stab}{Stab}
\DeclareMathOperator{\fix}{Fix}
\DeclareMathOperator{\Tr}{Tr}

\date{}

\title{The Unisolvence of Lagrange Interpolation with Symmetric Interpolation Space and Nodes in High Dimension}
\author[1,2]{Yulin Xie}
\author[1,2]{Yifa Tang\thanks{Corresponding author: tyf@lsec.cc.ac.cn}}
\affil[1]{LSEC, ICMSEC, Academy of Mathematics and Systems Science, Chinese Academy of Sciences, Beijing 100190, China}
\affil[2]{School of Mathematical Sciences, University of Chinese Academy of Sciences, Beijing 100049, China}

\begin{document}
\maketitle

\begin{abstract}
    High-dimensional Lagrange interpolation plays a pivotal role in finite element methods, where ensuring the unisolvence and symmetry of its interpolation space and nodes set is crucial. In this paper, we leverage group action and group representation theories to precisely delineate the conditions for unisolvence. We establish a necessary condition for unisolvence: the symmetry of the interpolation nodes set is determined by the given interpolation space. Our findings not only contribute to a deeper theoretical understanding but also promise practical benefits by reducing the computational overhead associated with identifying appropriate interpolation nodes.
\end{abstract}

\keywords{Lagrange interpolation, Unisolvence, Symmetry, Group action, Group representation, High dimension.}
\maketitle

\section{Introduction}
Given a finite-dimensional linear space consisting of functions and a set of points, Lagrange interpolation involves constructing a function $f$ called an interpolation function contained in the space for an arbitrary function $p$, such that $f$ coincides with $p$ at these points. If such $f$ exists and is unique, we say the interpolation problem is \emph{unisolvent}. The linear space is referred to as the interpolation space, and the points are termed nodes.

We review some standard facts on interpolation using univariate polynomials. When the interpolation space consists of polynomials of degree less than or equal to $n$, and the nodes are $n+1$ distinct points, this interpolation is unisolvent, and the interpolation function can be obtained using the Lagrange formula~\cite{mastroianni2008interpolation}. However, if the number of variables exceeds 1, merely having distinct points cannot guarantee unisolvence. In fact, if the interpolation space consists of continuous functions, then there exist distinct points in $\RR^n$ (for $n \geq 2$) such that interpolation is non-unisolvent~\cite{gasca2000polynomial}. Chung and Yao proposed a geometric characterization condition that is sufficient for unisolvence and provided corresponding Lagrange base functions~\cite{chung1977lattices}.

In finite element methods, Lagrange interpolation on an $n$-dimensional simplex is common, where interpolation function space and nodes are often symmetric with respect to barycentric coordinates~\cite{cui2017high}. Here a function space $F$ is symmetric means that if $f(x_1,x_2,\ldots,x_n)\in F$, then $f(x_{\sigma(1)},x_{\sigma(2)},\ldots,x_{\sigma(n)})\in F$ for each permutation $\sigma$ of $\{1,2,\ldots,n\}$. For a special family of polynomials on an $n$-dimensional simplex, corresponding nodes are constructed, termed regular nodes, which ensure interpolation problem is unisolvent~\cite{nicolaides1972class}. To achieve higher accuracy in numerical calculations, nodes must be adjusted while maintaining unisolvence. When $n=2,3,4$, a necessary condition for unisolvence is that the symmetry of nodes matches that of regular nodes~\cite{marchildon2022unisolvency}, and this condition holds for $n>4$ as well.~\cite{mulder2023unisolvence}

This paper proves a similar conclusion in more general situations. Specifically, if (1) the interpolation space and nodes are symmetric with coordinates(The specific definition will be provided in Section~\ref{section problem arising}); (2) the Lagrange interpolation is unisolvent, then the symmetry of nodes set is determined uniquely by the interpolation space. The conclusions prove the conjectures of ~\cite{mulder2023unisolvence} as corollaries. It should be pointed out that this conclusion holds not only for polynomial interpolation, but also for general interpolation problem. 

This article will be structured as follows: Section~\ref{section preliminaries} introduces fundamental concepts and theorems of group action and group representation theory essential for our analysis. In Section~\ref{section problem arising}, we reformulate our interpolation problem using the framework of group action, presenting a necessary and sufficient condition for its unisolvence. In Section~\ref{section equal number of orbits}, we establish a necessary and sufficient condition for the equivalence of the subset of $\RR^n$ under an action of $S_n$. We provide a necessary condition for unisolvence, thereby reducing the proof to demonstrating the invertibility of the matrix $V$. The Section~\ref{section V is positive definite} unveils a pivotal induced representation and demonstrates that $V$ is a Gram matrix. And we prove the conclusion through the linear independence of the characters. Finally, in Section~\ref{section conclusion}, we offer proofs for the two main theorems presented in this article, alongside prospects for future research endeavors.

\section{Preliminaries}\label{section preliminaries}

For better illustration, we recall some basic concepts on the action of groups and group representation theory. All these concepts and conclusions are standard and can be found in any group theory textbooks, e.g. ~\cite{rotman2012introduction}, ~\cite{serre1977linear}.

\subsection{A group acts on a set}

\begin{definition}\label{def group action}
    Let $G$ be a group and $X$ be a set. If a mapping $\phi : G \times X \to X$ satisfies the following conditions:

    \begin{enumerate}
        \item $\phi(g_1 g_2, x) = \phi(g_1, \phi(g_2, x))$ for all $g_1, g_2 \in G$ and $x \in X$,
    
        \item $\phi(e, x) = x$ for all $x \in X$, where $e$ is the unit element of $G$,
    \end{enumerate}

    then we call $\phi$ an action of the group $G$ on $X$.
\end{definition}

\begin{remark}
    Without causing ambiguity, we will omit $\phi$ and denote $\phi(g,x)$ briefly by $gx$. Using this notation, conditions 1 and 2 in Definition~\ref{def group action} become:
    
    \begin{enumerate}
        \item $(g_1 g_2)(x) = g_1(g_2x)$ for all $g_1, g_2 \in G$ and $x \in X$,
    
        \item $ex = x$ for all $x \in X$,
    \end{enumerate}

    In the sense, each element in $G$ can be regarded as a permutation of $X$, while the multiplication of $G$ is the composition of mappings, $e$ is the identity mapping, and the inverse mapping is the inverse in group $G$.

\end{remark}

\begin{definition}
    Let group $G$ act on sets $X$ and $Y$ respectively. We say the two actions are equivalent provided there exists a bijection $\psi$ from $X$ to $Y$, such that $g(\psi(x))=\psi(g(x))$ holds for each $g\in G$ and $x\in X$.
\end{definition}
\begin{remark}
    The equivalence defined above is indeed an equivalence relation among all sets acted on by group $G$.
\end{remark}

\begin{definition}
    If group $G$ acts on $X$ and $x\in X$, then the orbit of $x$ is $O(x):=\{gx:g\in G\}\subset X$.

    We often denote the orbit $O(x)$ by $Gx$. The orbits of $X$ form a partition. In fact, if we define $x\sim y$ by $y=gx$ for some $g\in G$, it is a equivalent relation in $X$ whose equivalent classes are the orbits.
\end{definition}

\begin{definition}
    If group $G$ acts on $X$, for each $x\in X$ the stabilizer of $x$ is defined by $\{g\in G:gx=x\}\subset G$. It is a subgroup of $G$ and we will use $\stab(x)$ to denote it. For each $g\in G$, the set $\{x\in X:gx=x\}\subset X$ is called the fixed points set of $g$, denoted by $\fix(g)$.
\end{definition}

\begin{lemma}[Burnside's Lemma]\label{theorem Burnside}
    For a finite group $G$ that acts on a finite set $X$, let $t$ be the number of orbits. Then, the Burnside's Lemma states that
    \[
    t=\frac{1}{|G|}\sum_{g\in G}|\fix(g)|
    \]
    The number of orbits is equal to the average of number of fixed points of elements of $G$.
\end{lemma}

\subsection{Group representation}
In this article, we will only consider the representation of groups in $\CC$.

\begin{definition}
    Let $V$ be a vector space over the filed $\CC$ of complex numbers and $GL(V)$ be the set of invertible linear mappings which becomes a group under the composite of mappings. 
    
    Let $G$ be a group. A group homomorphism $\phi:G\rightarrow GL(V)$ is called a linear representation of the group $G$, and $V$ is called the representation space of $\phi$.
\end{definition}

\begin{definition}
    Let $G$ acts on set $X$ and $V=\spn_\CC\{e_x\}_{x\in X}$ be a vector space indexed by the elements of $X$. For $g\in G$, let $\phi(g)$ be the linear mapping from $V$ to itself which sends $e_x$ to $e_{gx}$. In this way we obtain what we shall call the permutation representation.
\end{definition}

\begin{definition}
    Let $\phi:G\to GL(V)$ be a linear representation of a finite group $G$. For each $g\in G$, put $\chi_\phi(g):=\Tr(\phi(g))$ which is the trace of $\phi(g)$. The complex valued function $\chi_\phi$ on $G$ is called the character of the representation $\phi$. 
\end{definition}
\begin{remark}\label{remark character is class function}
    According to the property of trace $\Tr(AB)=\Tr(BA)$, we have that $\chi_\phi(ghg^{-1})=\chi_\phi(h)$ holds for $g,h\in G$. 
\end{remark}

\begin{definition}
    Let $G$ be a finite group. We call a complex valued function $f:G\rightarrow \CC$ on $G$ a class function if
    \[
    f(ghg^{-1})=f(h),\forall g,h\in G
    \]
    In other words, a class function takes a constant value on each conjugate class of group G. Let $\rm\bf H$ be the vector space of all class functions. If we define a scalar product $\langle\cdot,\cdot\rangle$ on $\rm\bf H$ by
    \[
    \langle f_1,f_2\rangle:=\frac{1}{|G|}\sum_{g\in G} f_1(g)\overline{f_2(g)}
    \]
    then $\rm\bf H$ becomes a Hilbert space.
\end{definition}
\begin{remark}
    Each character function is a class function by Remark~\ref{remark character is class function}. 
\end{remark}

\begin{definition}\label{def induce representation}
    Let $\phi:G\to GL(V)$ be a linear representation of group $G$, $H$ be a subgroup of $G$ and $W$ be a subspace of $V$ satisfying $\phi(h) W=W,\ \forall h\in H$. Here $\phi(h) W$ means the image of $W$ under $\phi(h)$. We define a representation $\theta:H\to GL(W)$ of $H$ by $\theta(h):=\phi(h)|_W$. If $V$ has a direct decomposition
    \[
    V=\bigoplus_{\sigma \in G/H} \phi(\sigma)W
    \]
    where $\sigma$ is a representative of each left coset of $H$, we call $\phi$ is induced by $\theta$.
\end{definition}
\begin{remark}
    In fact $\phi(\sigma)W$ is independent of the choice of $\sigma$ in coset $\sigma H$ since $\phi(h) W=W,\forall h\in H$.
\end{remark}

\begin{theorem}[Frobenius reciprocity formula]\label{theorem frobenius}
    Suppose $\phi:G\to V$ is induced by $\theta: H\to W$ and $f$ is a class function on $G$. Frobenius reciprocity formula states that
    \[
    \langle f_H,\chi_\theta\rangle_H=\langle f,\chi_\phi\rangle_G
    \]
    where $f_H$ is the restriction of f on H,  $\langle\cdot,\cdot\rangle_G$ and $\langle\cdot,\cdot\rangle_H$ represent the inner product on the class function space of $G$ and $H$, respectively.
\end{theorem}

\section{The symmetry and interpolation in $\RR^n$}\label{section problem arising}
\subsection{The unisolvence of Lagrange interpolation in $\RR^n$}
\begin{definition}
    Let $F$ be a vector space of functions from $\RR^n$ to $\RR$, and $a_1,a_2,\ldots,a_N$ be $N$ distinct points in $\RR^n$. The Lagrange interpolation means that for arbitrary real numbers $b_1,b_2,\ldots,b_N$, we want to find a function $f\in F$ such that $f(a_i)=b_i$ holds for $1\leq i\leq N$. If such $f$ exists and is unique, we say the Lagrange interpolation is unisolvent. $a_1,a_2,\ldots,a_N$ and $f\in F$ are called the interpolation nodes and interpolation function, respectively.
\end{definition}

According to classical interpolation theory, a necessary condition for unisolvence is that the dimension of interpolation space is equal to the number of interpolation nodes. If $\{f_1,f_2,\ldots,f_N\}$ is a basis of $F$, the following theorem gives some equivalent conditions of unisolvence.

\begin{theorem}\label{thm unisolvence equivalent condition}
    Let $\{f_1,f_2,\ldots,f_N\}$ be a basis of interpolation function space and $a_1,a_2,\ldots,a_N$ be interpolation nodes. The following four conditions are equivalent:

    (1)This Lagrange interpolation is unisolvent;

    (2)The determinant $\det (f_i(a_j))_{N \times N}$ is non-zero;

    (3)Suppose $f\in F$ satisfying $f(a_i)=0,1\leq i\leq N$, then $f=0$.

    (4)View $a_i$ as an element of the dual space of $F$(i.e. the linear function on $F$) in the following way: $a_i(f):=f(a_i)$, and then $a_1,a_2,\ldots,a_N$ is linearly independent.
\end{theorem}

\begin{proof}
    Since (1)-(3) are standard conclusions, we only need to prove that (4) is equivalent to (1)-(3).

    We consider a linear mapping $\phi: F \to F^*$, where $\phi$ sends $f_i$ to $a_i$. We note that the dimension of $F$ is equal to that of $F^*$, where $F^*$ is the dual space of $F$, so (4) is equivalent to $\phi$ being surjective and further, to $\phi$ being invertible.

    We choose $\{f_1,f_2,\ldots,f_N\}$ and its dual basis $\{f_1^*,f_2^*,\ldots,f_N^*\}$ as the basis of $F$ and $F^*$, respectively. Let $\Phi=(\Phi_{ij})$ be the matrix of $\phi$ under the two bases, by $\phi(f_1,f_2,\ldots,f_N)=(f_1^*,f_2^*,\ldots,f_N^*)\Phi$. We have for each $j$,
    \[a_j=f_1^*\Phi_{1j}+f_2^*\Phi_{2j}+\ldots+f_N^*\Phi_{Nj}
    \]
    Both sides of the equation taking values at $f_i$ shows that $f_i(a_j)=\Phi_{ij}$, hence the matrix $\Phi=(f_i(a_j))_{N\times N}$. The condition (2) means that $\phi$ is invertible, which is equivalent to (4).
\end{proof}

\subsection{$S_n$ acts on $\RR^n$}

Let $S_n$ be the set of all permutations of $\{1,2,\ldots,n\}$. That is to say, each $\sigma\in S_n$ is a bijection from $\{1,2,\ldots,n\}$ to itself. For each $\sigma\in S_n$, we can view it as a mapping $f(\sigma)$ from $\RR^n$ to $\RR^n$ in the following way: $f(\sigma)(x_1,x_2,\ldots,x_n):=(x_{\sigma(1)},x_{\sigma(2)},\ldots,x_{\sigma(n)})$. It is obvious that $f(\sigma)$ is an invertible linear mapping and satisfies that $f(\sigma)\circ f(\eta)=f(\sigma\eta)$ for $\sigma,\eta\in S_n$. Therefore $f$ provides not only an action of $S_n$ on $\RR^n$ but also a representation of $S_n$. For convenience, in the following text, we will omit $f$ and use $\sigma$ instead of $f(\sigma)$ when there is no ambiguities.

\begin{definition}\label{def symmetric subset of R^n}
    Suppose $X$ is a subset of $\RR^n$, we call it a symmetric subset if $\sigma x\in X,\forall \sigma \in S_n,x\in X$.
\end{definition}

According to the above definition, $\RR^n$ itself is a symmetric subset and each orbit is a symmetric subset under the action of $S_n$.

\subsection{$S_n$ acts on $f:\RR^n\to\RR$}

For a real-valued function $f$ on $\RR^n$, we consider the action of $S_n$ on it as follows: $\sigma f := f \circ \sigma^{-1}$. The reason $\sigma^{-1}$ is used here instead of $\sigma$ is to ensure that this is a group action. The verification is as follows:
\[
(\sigma\eta)(f) = f \circ (\sigma\eta)^{-1} = f \circ \eta^{-1} \circ \sigma^{-1} = \sigma (\eta(f))
\]
We actually obtain an action of $S_n$ on the set $\{f : \RR^n \to \RR\}$, and now we consider actions on its subsets.

\begin{definition}
    Let $F$ be a subset of $\{f:\RR^n\to\RR\}$. We call it symmetric if $\sigma f\in F,\ \forall \sigma\in S_n,f\in F$.
\end{definition}
Similar to the definition~\ref{def symmetric subset of R^n}, if $F$ is symmetric, we can consider the action of $S_n$ on $F$.

\subsection{Symmetric Lagrange interpolation}
Consider a Lagrange interpolation, where the interpolation space is $F$ and the interpolation nodes are $a_1,a_2,\ldots,a_N$. Assume that $F$ not only is symmetric, but also has a symmetric set of basis functions and $\{a_1,a_2,\ldots,a_N\}$ is a symmetric subset of $\RR^n$.

It should be noted that a symmetric interpolation function space does not necessarily have a symmetric basis function.

\begin{example}
    In $\RR^2$, where $x$ and $y$ are the coordinates, $F$ is defined as $\spn\{x-y\}$. It is clear that $y-x = -(x-y) \in F$, therefore $F$ is symmetric. However, it does not have a symmetric set of basis functions.
\end{example}

Our question is: what are the necessary conditions for the interpolation function and interpolation nodes to ensure the unisolvence of the interpolation problem? This article presents the following results:

\begin{enumerate}
    \item Given the interpolation function $F$, the symmetry of nodes must be unique; that is, if two nodes sets both satisfy unisolvence, then these two nodes sets are equivalent under the action of $S_n$.
    
    \item If the basis functions of $F$ satisfy certain properties, such as having symmetric monomial basis functions, then the basis functions and nodes set are equivalent under the action of $S_n$.
\end{enumerate}

We provide the following examples to assist in explanation.

\begin{example}\label{example n=3 p=2}
In $\RR^3$, where $F = \{x^2, y^2, z^2, xy, yz, zx\}$ is symmetric. Let $\spn \, F$ be the interpolation space, and we want to find symmetric interpolation nodes satisfying unisolvence.

There are three kinds of orbits in $\RR^3$: the first is a point with three identical coordinates, the second is a point with two identical coordinates but different from the third coordinate, and the third is a point with three distinct coordinates. The number of points in the orbits is $1, 3, 6$, respectively.

\begin{table}[h]
\begin{center}
\caption{The orbit classes in $\RR^3$}
\begin{tabular}{lllll}
\cline{1-3}
orbit class & \multicolumn{1}{c}{orbit}                                                                          & number of elements &  &  \\ \cline{1-3}
class 1     & $\{(a,a,a)\}$                                                                                      & 1                  &  &  \\ \cline{1-3}
class 2     & $\{(a,a,b),(a,b,a),(b,a,a)\}$                                                                      & 3                  &  &  \\ \cline{1-3}
class 3     & \begin{tabular}[c]{@{}l@{}}$\{(a,b,c),(a,c,b),(b,a,c),$\\ $(b,c,a),(c,a,b),(c,b,a)\}$\end{tabular} & 6                  &  &  \\ \cline{1-3}
\end{tabular}
\end{center}
\end{table}

Unisolvence requires that the number of nodes is equal to the dimension of the interpolation space. Therefore, all possible combinations are: six class-1 orbits; three class-1 orbits and one class-2 orbit; two class-2 orbits; one class-3 orbit.

In Case 1, suppose $P = \{(a_i, a_i, a_i) : i = 1, 2, 3, 4, 5, 6\}$. By calculating the determinant in Theorem~\ref{thm unisolvence equivalent condition}(2), we find that it is always non-unisolvent regardless of the value of $a_i$, and the same applies to Cases 2 and 4.

In Case 3, suppose $P = \{(a, a, b), (a, b, a), (b, a, a), (c, c, d), (c, d, c), (d, c, c)\}$ where $a, b, c, d \in \RR$. By Theorem~\ref{thm unisolvence equivalent condition}(2), the interpolation is unisolvent if and only if $(a-b)^2(c-d)^2(ad-bc)^2(4ac-ad-bc-2bd)\neq 0$.
\end{example}

The above example illustrates that, given the interpolation function space, in order to ensure unisolvence, the symmetry of the interpolation nodes should also be determined. If we consider the set of basis functions and the set of interpolation nodes, we will find that they are actually equivalent under the action of $S_n$ . The bijection can be given as follows:
\[
\begin{aligned}
    f(a,a,b)=z^2,f(a,b,a)=y^2,f(b,a,a)=x^2\\
    f(c,c,d)=xy,f(c,d,c)=zx,f(d,c,c)=yz
\end{aligned}
\]

\section{A necessary condition for nodes set}\label{section equal number of orbits}
\subsection{The type and equivalence of orbits in $\RR^n$}\label{subsection orbit class}

The set $\RR^n$ can be divided into disjoint orbits under the action of $S_n$. We distinguish orbits by defining their types.

Given an orbit $O$, let $x=(x_1,x_2,\ldots,x_n)$ be one element of $O$. We divide $x_1,x_2,\ldots,x_n$ into several equivalence classes based on whether the components are equal. Components in the same equivalence class are equal, while component values in different equivalence classes are different. The number of equivalence classes with $i$ elements is denoted by $c_i$, where $1 \leq i \leq n$. For example, $c_1$ is the number of components in $x$ that are different from the other components.
\begin{definition}
    We refer to $(c_1, c_2,\ldots, c_n)$ as the type of the orbit $O$.
\end{definition}
\begin{remark}
    Noting that we define the type based on one element in the orbit, however, according to the definition of the action of $S_n$, the elements in the same orbit are only exchanged in the order of their components, so this definition is independent of the selection of $x$.
\end{remark}
In Example~\ref{example n=3 p=2}, the types of the three orbits are $(0,0,1)$, $(1,1,0)$, and $(3,0,0)$.

Because there are a total of $n$ components, the equation $c_1 + 2c_2 + \cdots + nc_n = n$ holds, where $c_i$ is a non-negative integer. Conversely, given any non-negative integer solution $(c_1, c_2, \ldots, c_n)$ to the Diophantine equation $c_1 + 2c_2 + \cdots + nc_n = n$, we can construct an orbit such that its type is exactly $(c_1, c_2, \ldots, c_n)$. Therefore, there exists a bijection between the types of orbits in $\RR^n$ and the non-negative integer solutions of the Diophantine equation $c_1 + 2c_2 + \cdots + nc_n = n$. Thus, the number of orbits classes in $\RR^n$ is finite, equal to the number of non-negative integer solutions of the Diophantine equation $c_1 + 2c_2 + \cdots + nc_n = n$.

We have the following theorem regarding the type and equivalence of orbits.
\begin{theorem}
    Under the action of $S_n$, two orbits in $\RR^n$ is equivalent if and only if they have the same type.
\end{theorem}
\begin{proof}
    Suppose two orbits $O$ and $\Tilde{O}$ have the same type $(c_1,c_2,\ldots,c_n)$. We choose a representative element $x$ contained in $O$ such that its components arrange in the following way: the first $c_1$ are the components which are different from all other components, followed by $2c_2$ components which are all equal pairs of components, and so on. For example, $(b,c,d,d,c,d,a)$ needs to be arranged as $(a,b,c,c,d,d,d)$. Because $O$ contains all the component arrangements of $x$, such $x$ can always be found.

    Let $\sigma\in\stab(x)$, the stabilizer of $x$. $\sigma$ needs to keep the first $c_1$ components fixed, and each component can only exchange positions with its own equal component.

    We define the action of $S_n$ on $S_n/\stab(x)$, all left cosets of $\stab(x)$ as $\sigma (\eta \stab(x))=(\sigma\eta) \stab(x)$. We claim that $O$ is equivalent to $S_n/\stab(x)$. 
    
    In fact, for $\sigma x\in O$, we define $f(\sigma x)=\sigma \stab(x)\in S_n/\stab(x)$. Since $\sigma x =\eta x$ implies $\eta^{-1}\sigma\in \stab(x)$, which again implies $f(\sigma x)=\sigma \stab(x)=\eta \stab(x)=f(\eta x)$, $f$ is well defined. We check $O$ is equivalent to $S_n/\stab(x)$ under $f$: 
    \[
    \eta(f(\sigma x))=\eta(\sigma \stab(x))=(\eta\sigma) \stab(x)=f(\eta \sigma(x)).
    \]

    We use the same method to select a representative element $y$ from $\Tilde{O}$ and then $\stab(x)=\stab(y)$. The above argument shows that $O$ and $\Tilde{O}$ are equivalent to $S_n/\stab(x)$ and $S_n/\stab(y)$, respectively. Hence $O$ and $\Tilde{O}$ are equivalent based on transitivity.

    Conversely, suppose orbits $O$ and $\Tilde{O}$ are equivalent and $f:O\to\Tilde{O}$ is the bijection. The type of $O$ is $(c_1,c_2,\ldots,c_n)$. We choose $x$ as a representative element of $O$ in the same way. We have $\stab(x)=\stab(f(x))$:
    \[
    \begin{aligned}
        \sigma\in\stab(f(x))
        &\iff \sigma f(x)=f(x)
        \iff f(\sigma x)=f(x)\\
        &\iff \sigma x=x
        \iff \sigma\in\stab(x).
    \end{aligned}
    \]
    This shows that each $\sigma\in\stab(f(x))$ keeps the first $c_1$ components fixed, and exchanges $i$-th and $(i+1)$-th component for $i=c_1+1,c_1+3,\ldots,c_1+2c_2-1$, and so on. Thus the type of $f(x)$ which is the type of $\Tilde{O}$ is also $(c_1,c_2,\ldots,c_n)$. This competes the proof.
\end{proof}

For the convenience of the following description, we define an order among all orbit classes in $\RR^n$. Let $(c_1,c_2,\ldots,c_n)$ and $(\Tilde{c}_1,\Tilde{c}_2,\ldots,\Tilde{c}_n)$ be types of different orbit classes. We define $(c_1,c_2,\ldots,c_n)>(\Tilde{c}_1,\Tilde{c}_2,\ldots,\Tilde{c}_n)$ if there exists $1\leq k\leq n$ such that $c_k>\Tilde{c}_k$ and $c_l=\Tilde{c}_l$ for all $l>k$. That is to say, when we compare two types, we compare from high $c_n$ to low $c_1$. It is a total order among all orbits classes.

\subsection{The equivalence between finite symmetric subsets}

Let $X$ be a finite symmetric subset of $\RR^n$. $X$ can be divided into the union of disjoint orbits. We will denote by $c$ the number of orbit classes in $\RR^n$, where $c$ is the number of non-negative integer solution of Diophantine equation $c_1 + 2c_2 + \cdots + nc_n = n$. We get a vector $(X_1,X_2,\ldots,X_c)\in \NN^c$, called the orbit vector of $X$, where each component $X_i$ is a non-negative integer that is the number of the $i$-th orbit classes $X$ has. The order is in descending order according to the type of orbit classes. We have a theorem that is used to determine whether two symmetric subsets are equivalent.

\begin{theorem}\label{theorem condition for equivalent symmetric subsets}
    Let $X$ and $Y$ be two finite symmetric subsets. We assume that their orbit vectors are $(X_1, X_2, \dots, X_c)$ and $(Y_1, Y_2, \dots, Y_c)$, respectively. Then, a sufficient and necessary condition for the equivalence of $X$ and $Y$ is $X_i = Y_i$, $1 \leq i \leq c$.
\end{theorem}
\begin{proof}
    The necessity of the theorem is evident, so we only prove sufficiency. Suppose $X_i = Y_i$, $1 \leq i \leq c$. We decompose $X$ and $Y$ into the union of several orbits: $X = \bigcup_{i=1}^t U_i$ and $Y = \bigcup_{i=1}^t V_i$.

    Since the numbers of orbit classes of $X$ and $Y$ are the same, we can assume that $U_i$ and $V_i$ are equivalent for each $i$, $1 \leq i \leq t$. That is, there exists a bijection $f_i: U_i \to V_i$ such that $f_i(\sigma x) = \sigma(f_i(x))$ for all $x \in U_i$ and $\sigma \in S_n$. Then, we can define a bijection $f: X \to Y$, by letting $f(x) := f_i(x)$ for $x\in U_i$. With this definition, $f$ is indeed a bijection.

    Let $x \in U_i$, then $\sigma(x) \in U_i$, and 
    \[
    f(\sigma(x)) = f_i(\sigma(x)) = \sigma(f_i(x)) = \sigma(f(x)).
    \]
    
    This shows that $X$ and $Y$ are indeed equivalent under the action of $S_n$.
\end{proof}

According to the above theorem, to demonstrate the unique symmetry of the interpolation nodes set under the action of $S_n$, it suffices to show that the interpolation nodes set has a unique orbit vector. Below, we will prove that such an orbit vector is indeed unique through different linear constraints.

\subsection{unisolvence and group action}
In Example~\ref{example n=3 p=2}, to ensure unisolvence, we note that the number of orbits for symmetric basis functions and symmetric interpolation nodes set are the same. This is not a coincidence. In fact, the following theorem holds.

\begin{theorem}\label{theorem number of orbit equal}
    Suppose the symmetric interpolation function space $F$ has a symmetric basis $\{f_1, f_2, \ldots, f_N\}$, and the symmetric interpolation nodes set is $\{a_1, a_2, \ldots, a_N\}$. If the interpolation problem is unisolvent, then the number of orbits under the action of $S_n$ for the basis functions and interpolation nodes set are equal.
\end{theorem}
\begin{proof}
    Let $\{O_1,O_2,\ldots,O_s\}$ be orbits of $\{f_i,1\leq i\leq N\}$ and $\{\Tilde{O}_1,\Tilde{O}_2,\ldots,\Tilde{O}_t\}$ be that of $\{a_i,1\leq i\leq N\}$. We argue by contradiction.

    Suppose $s>t$. For $1\leq i\leq s$, we define $h_i=\sum\limits_{f\in O_i} f\in F$. For $1\leq j\leq t$, let $b_j$ be a representative of $\Tilde{O}_j$. Consider a linear system as follows
    \[
    \begin{bmatrix}
        h_1(b_1) & h_2(b_1) & \cdots & h_s(b_1)\\
        h_1(b_2) & h_2(b_2) & \cdots & h_s(b_2)\\
        \vdots   & \vdots   & \ddots & \vdots  \\
        h_1(b_t) & h_2(b_t) & \cdots & h_s(b_t)\\
    \end{bmatrix}
    \begin{bmatrix}
        \alpha_1\\
    \alpha_2\\
    \vdots\\
    \alpha_s
    \end{bmatrix}
    =0.
    \]
    On account of $s>t$, it has a nontrivial solution $[\alpha_1,\alpha_2,\ldots,\alpha_s]^T$. Set $h=\sum\limits_{i=1}^s \alpha_ih_i$, then $h\neq 0\in F$. We claim that $h(a_i)=0,1\leq i\leq N$. Indeed for each $a_i$ there exists $b_j$ and $\sigma \in S_n$ such that $a_i=\sigma(b_j)$ because each $a_i$ is in some orbit. Hence
    \[
    h(a_i)=h\circ\sigma(b_j)=\sum_{k=1}^s \alpha_k (h_k\circ\sigma)(b_j)=\sum_{k=1}^s \alpha_kh_k(b_j)=0.
    \]
    The third equality holds because $\sigma$ is a permutation on $O_k$ which yields $h_k\circ \sigma=\sum\limits_{f\in O_k}f\circ \sigma=\sum\limits_{f\in O_k}f=h_k$. The last equality holds since $[\alpha_1,\alpha_2,\ldots,\alpha_s]^T$ is a solution of the linear system. Theorem~\ref{thm unisolvence equivalent condition}(3) now shows a contradiction.

    Suppose $s<t$. By a similar argument, we can prove that there exists some $a$ which is a non-zero linear combination of $a_j$ such that $a(f_i)=0,1\leq i\leq N$. This gives that $a=0$ in the dual space of $F$. By Theorem~\ref{thm unisolvence equivalent condition}(4), it is contrary to the assumption.

    Therefore, only $s=t$ can hold, which proves the theorem.
\end{proof}
\begin{remark}
    Note that in the proof of the above theorem, we only used the action of $S_n$ on the interpolation basis functions and interpolation nodes. Therefore, if we replace $S_n$ with any of its subgroup, the theorem still holds.
\end{remark}

Suppose the interpolation nodes set $\{a_1, a_2, \ldots, a_N\}$ forms a finite symmetric subset of $\RR^n$, with orbit vector $(X_1, X_2, \ldots, X_c)$. According to the definition of orbit vector, we have that the number of orbits for the interpolation nodes set equals $X_1 + X_2 + \cdots + X_c$. According to the above theorem, this must be equal to the number of orbits for the basis vectors. This provides a linear constraint on the orbit vectors.

For an orbit in $\RR^n$ under the action of $S_n$, if we replace the action of $S_n$ with the action of a subgroup $H$ of $S_n$, the orbit will be decomposed into smaller orbits. For example, suppose $x$ and $y$ belong to the same orbit $O$ under the action of $S_n$, i.e., there exists $g \in S_n$ such that $y = g(x)$. However, there might not exist such $g$ in $H$, leading to them belonging to different orbits. Then the orbit $O$ will be decomposed into different orbits contains $x$ and $y$ respectively. 

By replacing the action of $S_n$ with that of a subgroup of $S_n$ on the basis functions and interpolation nodes, we will obtain additional linear constraints. But before proceeding to this point, we need the following theorem about orbital classes.

\begin{theorem}
    Let $X$ and $Y$ be two orbits in $\RR^n$ under the action of $S_n$. Suppose $H=\stab(x)$ is the stabilizer of an element $x$ in $X$, and $v$ is the number of orbits of $Y$ under the action of $H$. Then, $v$ depends only on the orbit classes to which $X$ and $Y$ belong, and is independent of the choice of $x$.
\end{theorem}
\begin{proof}
    First, we prove that $v$ is independent of the choice of $x$:

    Suppose $x$ and $\tilde{x}=\sigma(x)$ are two elements in $X$. Then, 
    \[
    \eta \in \stab(\tilde{x}) \iff \eta(\tilde{x})=\tilde{x} \iff \eta\sigma(x)=\sigma(x) \iff \sigma^{-1}\eta\sigma\in \stab(x).
    \]
    Hence, we have $\stab(\tilde{x})=\sigma \stab(x)\sigma^{-1}$.

    Combined with Theorem~\ref{theorem Burnside}, the number of orbits of $Y$ under the action of $\stab(\tilde{x})$ is
    \[
    \frac{1}{|\stab(\tilde{x})|}\sum_{\eta\in \stab(\tilde{x})} |\fix(\eta)|=\frac{1}{|\stab(x)|}\sum_{\eta\in \stab(x)} |\fix(\sigma\eta\sigma^{-1})|
    \]
    It's easy to verify that $\fix(\sigma\eta\sigma^{-1})=\sigma\fix(\eta),\ \forall \eta,\sigma\in S_n$, and since $\sigma$ is a bijection on $Y$, $|\fix(\sigma\eta\sigma^{-1})|=|\fix(\eta)|$.

    According to Theorem~\ref{theorem Burnside}, the above expression is also the number of orbits of $Y$ under the action of $\stab(x)$. Therefore, the number of orbits is independent of the choice of $x$.

    Now, let $X_1$ and $X_2$ be two equivalent orbits. Let $f:X_1\to X_2$ be a bijection. Suppose $x_1$ and $x_2=f(x_1)$ are elements in $X_1$ and $X_2$, respectively. Then,
    \[
    \begin{aligned}
        \sigma\in \stab(x_2) 
        &\iff \sigma(x_2)=x_2 
        \iff\sigma(f(x_1))=f(x_1) 
        \iff f(\sigma x_1)=f(x_1)\\
        &\iff \sigma x_1=x_1 
        \iff \sigma\in \stab(x_1)
    \end{aligned}
    \]
    Hence, $\stab(x_1)=\stab(x_2)$. Under the same group action, the number of orbits must be the same.

    Additionally, suppose $Y_1$ and $Y_2$ are equivalent under the action of $S_n$. Then, they are also equivalent under the action of any subgroup of $S_n$. Thus, the number of orbits remains the same.
\end{proof}

According to the above theorem, we can define $v_{ij}$ as the number of orbits generated by the stabilizer of an element in orbit class $i$ acting on orbit class $j$. Finally, we obtain the following theorem.

\begin{theorem}\label{theorem VX=r}
    The assumption is the same as Theorem~\ref{theorem number of orbit equal}. Set the orbit vector of the nodes set is $(X_1,X_2,\ldots,X_c)$, then the following linear system holds:
    \[
    \begin{bmatrix}
        v_{11} & \cdots & v_{1c}\\
        \vdots & \ddots & \vdots\\
        v_{c1} & \cdots & v_{cc}
    \end{bmatrix}
    \begin{bmatrix}
        X_1\\
        \vdots\\
        X_c
    \end{bmatrix}
    =
    \begin{bmatrix}
        r_1\\
        \vdots\\
        r_c
    \end{bmatrix}
    \]
    where $r_i$ is the number of orbits of basis functions under the action of the stabilizer of some element in orbit class $i$.
\end{theorem}
\begin{proof}
    In Theorem~\ref{theorem number of orbit equal}, we replace $S_n$ with the stabilizer of some element in orbit class $i$. Then we have the number of orbits of nodes set is $v_{i1}X_1+v_{i2}X_2+\ldots+v_{ic}X_c$ which is equal to that of basis functions $r_i$ by definition. It is the $i$-th linear equation.
\end{proof}
\begin{remark}
    There is a similar conclusion in \cite{mulder2023unisolvence}, but lacks a strict proof and sufficient reasons.
\end{remark}

If the matrix $V=(v_{ij})_{c\times c}$ in the above theorem is invertible, then multiplying both sides by $V^{-1}$ yields the desired result: given a set of basis vectors, the symmetry of the interpolation nodes set is unique. We will prove this fact in the following section.

For example, when $n=3$, the types of the three orbits are $(0,0,1)>(1,1,0)>(3,0,0)$. Based on this, the calculated matrix $V$ is
\[
\begin{bmatrix}
    1 & 1 & 1\\
    1 & 2 & 3\\
    1 & 3 & 6
\end{bmatrix}
\]

Calculating the determinant $\det(V)=1$, so the matrix is invertible.

\section{Symmetry and reversibility of matrix V}\label{section V is positive definite}

\subsection{$V$ is symmetric}

Let $v_{ij}$ represent the number of orbits obtained by the action of the stabilizer $H_i$ of an element in orbit class $i$ on the orbit class $j$. Let $H_i$ and $H_j$ be the stabilizer of elements in orbit classes $i$ and $j$, respectively. According to the knowledge of group actions, we know that the action of $S_n$ on the left cosets $S_n/H_j$ is transitive and belongs to the orbit class $j$. Therefore, we obtain another representation of $v_{ij}$: $v_{ij}$ represents the number of orbits obtained by the action of $H_i$ on all left cosets $S_n/H_j$. The following theorem reveals the symmetry of $v_{ij}$.

\begin{theorem}\label{theorem V is gram matrix}
    Let $H_i$ be the stabilizer of an element in orbit class $i$. Let $\chi_i$ denote the character of the permutation representation obtained by the action of $S_n$ on the left cosets $S_n/H_i$. Then, $v_{ij} = \langle \chi_i, \chi_j \rangle$.
\end{theorem}

To prove Theorem~\ref{theorem V is gram matrix}, we need some preliminaries.

Consider two subgroups $H_i$ and $H_j$ of the permutation group $S_n$. $S_n$ has a canonical action on the set of left cosets $S_n/H_i$: for $\sigma \in S_n$ and $\eta H_i$ a coset, define $\sigma (\eta H_i) := (\sigma \eta)H_i$. Similarly, we can define the action of $S_n$ on $S_n/H_j$. Furthermore, we can consider the action of $S_n$ on the Cartesian product $(S_n/H_i) \times (S_n/H_j)$, defined by acting separately on each component.

$H_i$ naturally acts on the subset $\{H_i\} \times (S_n/H_j)$ of the Cartesian product $(S_n/H_i) \times (S_n/H_j)$. The following lemma reveals the relationship between these two actions.

\begin{lemma}\label{lemma indece representation}
    Let $H_i$ and $H_j$ be two subgroups of $S_n$. Suppose $V_i$ and $V_j$ are free linear spaces generated by the left coset decompositions of $S_n/H_i$ and $S_n/H_j$, respectively. Let $\rho: G \to GL(V_i \otimes V_j)$ and $\theta: H_i \to GL(\spn\{H_i\} \otimes V_j)$ be the permutation representations as described above.

    Then, the representation $\rho$ is the induced representation of $\theta$.
\end{lemma}
\begin{proof}
    Firstly, we need to check that $\spn\{H_i\}\otimes V_j$ is an invariant subspace under $H_i$. To do this, note that $hH_i=H_i$ for all $h\in H_i$. Therefore, for any vector $aH_i\otimes v_j \in \spn\{H_i\}\otimes V_j$ and any $h\in H_i$, we have $h(aH_i\otimes v_j) = h(aH_i)\otimes hv_j = aH_i\otimes hv_j \in \spn\{H_i\}\otimes V_j$.

    Secondly, let $R$ be a set of representatives chosen from each left coset of $H_i$. Then, we have $S_n/H_i=\{rH_i:r\in R\}$. Furthermore, we obtain the direct sum decomposition: $V_i=\bigoplus_{r\in R}\spn\{rH_i\}$, since $V_i$ is a free linear space generated by $S_n/H_i$. According to the distributive law, we have the following decomposition:
    \[
    V_i\otimes V_j = \bigoplus_{r\in R}(\spn \{rH_i\}\otimes V_j) = \bigoplus_{r\in R}\left(\rho(r)\left(\spn \{H_i\}\otimes V_j\right)\right),
    \]
    which completes the proof due to Definition~\ref{def induce representation}.
\end{proof}

Next we give a proof of Theorem~\ref{theorem V is gram matrix}.
\begin{proof}
    Since $H_i$ acts on $S_n/H_j$, according to Theorem~\ref{theorem Burnside}, we have
    \[
    v_{ij}=\frac{1}{|H_i|}\sum_{\sigma\in H_i}|\fix(\sigma)|.
    \]
    Note that the $\theta$ defined in Lemma~\ref{lemma indece representation} is actually the permutation representation of the above action. Therefore, the number of fixed points of $\sigma$ is precisely the character $\chi_\theta(\sigma)$. According to Lemma~\ref{lemma indece representation} and Theorem~\ref{theorem frobenius}, we have
    \[
    v_{ij} = \frac{1}{|H_i|}\sum_{\sigma\in H_i} \chi_\theta(\sigma) = \langle \chi_\theta, \mathbf{1} \rangle_{H_i} = \langle \chi_\rho, \mathbf{1} \rangle_{S_n},
    \]
    where $\mathbf{1}$ represents the class function that takes the value $1$ for all elements in $S_n$. Moreover, $\rho$ is actually the permutation representation of the product of two actions, and its character is $\chi_\rho = \chi_i \cdot \chi_j$.~\cite{serre1977linear} Therefore,
    \[
    v_{ij} = \langle \chi_\rho, \mathbf{1} \rangle_{S_n} = \frac{1}{|S_n|}\sum_{\sigma\in S_n}\chi_i(\sigma)\chi_j(\sigma)=\langle \chi_i, \chi_j \rangle_{S_n}.
    \]
\end{proof}
\begin{corollary}\label{corollary V is symmetric}
    The matrix $V$ is a Gram matrix, hence symmetric.
\end{corollary}

\subsection{$V$ is invertible}

Since the matrix $V$ is a Gram matrix, its invertibility is equivalent to the linear independence of $\chi_i$, which further leads us to consider the space of class functions on $S_n$. Class functions are those taking constant values on the conjugate classes of $S_n$.

Any permutation in $S_n$ can be decomposed into a product of disjoint cycles, and this decomposition is unique up to order. For a permutation $\sigma \in S_n$, we denote by $c_i$ the number of cycles of length $i$ in the cycle decomposition of $\sigma$. The tuple $(c_1, c_2, \dots, c_n)$ is called the type of $\sigma$, and two permutations are conjugate if and only if they have the same type \cite{rotman2012introduction}. Thus, there is a one-to-one correspondence between the conjugate classes in $S_n$ and the non-negative integer solutions of the Diophantine equation $c_1 + 2c_2 + \dots + nc_n = n$.

We denote the orbit classes in $S_n$, ordered from largest to smallest as defined in Section~\ref{subsection orbit class}, by $A_1, A_2, \dots, A_c$. This correspondence precisely matches the orbit classes in $\RR^n$, which helps us prove the following theorem.

\begin{theorem}\label{theorem V is invertible}
    The class functions $\chi_i$, for $1 \leq i \leq c$, defined in Theorem~\ref{theorem V is gram matrix}, are linearly independent. Therefore, the Gram matrix $V$ is invertible; moreover, it is positive definite.
\end{theorem}
\begin{proof}
    Let $\chi_i$ take values $K_{ij}$ on the conjugate class $A_j$. Notice that the number of $\chi_i$'s is the same as the number of conjugate classes $A_j$. Therefore, we only need to show that $K=(K_{ij})$ is a non-singular matrix.

    Let $\chi_i$ and $A_j$ have types $(c_1, c_2, \dots, c_n)$ and $(\tilde{c}_1, \tilde{c}_2, \dots, \tilde{c}_n)$, respectively. Suppose $i<j$, then there exists $1\leq k\leq n$ such that $c_k<\tilde{c}_k$ and $c_l=\tilde{c}_l$ for all $l>k$. By previous definitions, $\chi_i$ represents the character of the permutation representation of $S_n$ acting on the orbit class with type $(c_1, c_2, \dots, c_n)$, and hence the character of a permutation representation is the number of fixed points. Thus, for a permutation $\sigma$ in the conjugate class $A_j$, $K_{ij}$ represents the number of fixed points under the action of $\sigma$ on that orbit class. We claim that under the action of $\sigma$, this orbit class has no fixed point.

    Indeed, suppose $x=(x_1, x_2, \dots, x_n)$ is a fixed point, partitioned into equivalence classes according to equality relations, with $c_i$ $i$-element equivalence classes. Since $\sigma$ fixes $x$, $i$-cycles in $\sigma$ must act on equivalence classes with at least $i$ equal elements. According to the assumption, when $l>k$, the number of $l$-element equivalence classes is equal to the number of $l$-cycles. Therefore, $l$-cycles exhaust all $l$-element equivalence classes. Since $\tilde{c}_k$ exceeds the number of $k$-element equivalence classes, at least one $k$-cycle will cause changes in the components, contradicting $x$ being a fixed point. Hence, the assumption is false, and $K_{ij}=0$ when $j>i$.

    For $i=j$, the proof is similar to the above. The $i$-cycles of permutation $\sigma$ and the number of $i$-element equivalence classes of $x$ are equal. This implies that at least one $x$ is a fixed point of $\sigma$, so $K_{ii}\geq 1$. Thus, $K$ is a lower triangular matrix with non-zero diagonal entries, making it non-singular.
\end{proof}
\begin{remark}
    The above Theorem and Corollary~\ref{corollary V is symmetric} prove the conjucture $1$ of \cite{mulder2023unisolvence}.
\end{remark}

\section{Conclusion and future work}\label{section conclusion}
\subsection{Conclusion}
\begin{theorem}\label{theorem conclusion 1}
    Let the interpolation space $F$ have symmetric basis functions $f_i, 1\leq i\leq N$. Let $\{a_i: 1\leq i\leq N\}$ and $\{\tilde{a}_i: 1\leq i\leq N\}$ be symmetric subsets of $\RR^n$, serving as interpolation nodes, both unisolvent. Then, the actions of $S_n$ on $\{a_i: 1\leq i\leq N\}$ and $\{\tilde{a}_i: 1\leq i\leq N\}$ are equivalent.
\end{theorem}
\begin{proof}
    Let the orbit vectors of the interpolation nodes set be $X=(X_1,X_2,\ldots,X_c)^T$ and $\tilde{X}=(\tilde{X}_1,\tilde{X}_2,\ldots,\tilde{X}_c)^T$. According to Theorem~\ref{theorem VX=r}, we have $VX=V\tilde{X}$, and according to Theorem~\ref{theorem V is invertible}, the matrix $V$ is invertible. Therefore, $X=\tilde{X}$. By Theorem~\ref{theorem condition for equivalent symmetric subsets}, we complete the proof.

\end{proof}
In the above theorem, if the orbit class of the basis functions under the action of $S_n$ is a subset of the orbit classes in $\RR^n$, then we can further obtain the following theorem.
\begin{theorem}\label{theorem conclusion 2}
    Let the interpolation space $F$ have symmetric basis functions $f_i, 1\leq i\leq N$, and let $\{a_i:1\leq i\leq N\}$ be a symmetric subset of $\RR^n$, serving as interpolation nodes, where the interpolation problem is unisolvent. Suppose further that the orbit class of the basis functions is a subset of the orbit classes in $\RR^n$. Then, the action of $S_n$ on the basis functions and interpolation nodes is equivalent.
\end{theorem}
\begin{proof}
    Let the orbit vectors of the interpolation nodes and basis functions be $X=(X_1,X_2,...,X_c)^T$ and $Y=(Y_1,Y_2,...,Y_c)^T$, respectively. According to Theorem~\ref{theorem VX=r}, we have $VX=r=VY \implies X=Y$. Then, according to Theorem~\ref{theorem condition for equivalent symmetric subsets}, we complete the proof.
\end{proof}
\begin{remark}
    For certain basis functions, their orbits are not within the aforementioned orbit classes. For example, when $n=3$, the symmetric basis functions $\{x_1^2x_2+x_2^2x_3+x_3^2x_1, x_1^2x_3+x_3^2x_2+x_2^2x_1\}$ are not equivalent to any orbit class. This is because in $\RR^3$, there are only three types of orbit classes, as shown in Example~\ref{example n=3 p=2}, with number of elements $1$, $3$, and $6$, respectively, whereas the elements of this number of elements is $2$. This illustrates the meaningfulness of the conditions in the theorem.
\end{remark}
\begin{corollary}
    Assuming the same conditions as in Theorem~\ref{theorem conclusion 2}, if the basis functions are all monomials such as $x_i^{i_1}x_2^{i_2}...x_n^{i_n}$, then the actions of $S_n$ on the basis functions and interpolation nodes set are equivalent.
\end{corollary}

The above corollary proves the conjecture $2$ of \cite{mulder2023unisolvence}.

\subsection{Future work}

In symmetric high-dimensional interpolation, we have already provided a necessary condition for unisolvence. Inspired by one-dimensional polynomial interpolation, we believe it is possible to construct a condition that is both necessary and sufficient. This may require an in-depth exploration of the permutation representation in Lemma~\ref{lemma indece representation}. We set this as an open problem.

\bibliographystyle{abbrv}
\bibliography{arxiv}
\end{document}